\documentclass[reqno,final]{elsarticle}
\usepackage[color]{showkeys}
\definecolor{refkey}{rgb}{0,1,1}
\definecolor{labelkey}{rgb}{1,0,0}

\usepackage{amsmath,amssymb,amsthm,amsfonts}
\usepackage{tikz}
\usetikzlibrary{cd}
\usepackage{color,soul}
\usepackage{caption}
\usepackage{enumitem}

\journal{LAA}

\newcommand{\eq} [1] {\begin{equation}\label{#1}\quad}
\newcommand{\en} {\end{equation}}
\newcommand{\C}{\mathbb{C}}
\newcommand{\R}{\mathbb{R}}

\newcommand{\conv}{\operatorname{conv}}
\newcommand{\diag}{\operatorname{diag}}

\newcommand{\Span}{\operatorname{span}}
\newcommand{\tr}{\operatorname{tr}}
\newcommand{\re}{\operatorname{Re}}
\newcommand{\im}{\operatorname{Im}}
\newcommand{\DW}{\operatorname{DW}}

\newcommand{\NNR}{F_N}
\newcommand{\JNR}{\DW}

\newcommand{\abs}[1]{\left\vert#1\right\vert}
\newcommand{\norm}[1]{\left\Vert#1\right\Vert}
\theoremstyle{plain}
\newtheorem{theorem}{Theorem}[section]
\newtheorem{lemma}[theorem]{Lemma}
\newtheorem{proposition}[theorem]{Proposition}
\newtheorem{corollary}[theorem]{Corollary}

\theoremstyle{remark}
\newtheorem{remark}[theorem]{Remark}

\theoremstyle{definition}
\newtheorem{example}[theorem]{Example}

\begin{document}
\title{The normalized numerical range and the Davis-Wielandt shell}
\author[H]{Brian Lins\corref{cor}}
\date{}
\address[H]{Department of Mathematics and Computer Science \\
Hampden-Sydney College, Hampden Sydney VA 23943, USA}
\ead{blins@hsc.edu}

\tnotetext[support]{The results are partially based on the Capstone project of the third named under the supervision of the second named author. The latter was also supported in part by the Faculty Research funding from the Division of Science and Mathematics, New York University Abu Dhabi.}
\author[A]{Ilya M. Spitkovsky}

\address[A]{Division of Science, New York  University Abu Dhabi (NYUAD)\\ Saadiyat Island,
P.O.~Box 129188 Abu Dhabi, UAE}
\ead{ims2@nyu.edu, imspitkovsky@gmail.com}
\author[A]{Siyu Zhong}
\ead{sz1152@nyu.edu}

\cortext[cor]{Corresponding author.}



\begin{abstract}
{For a given $n$-by-$n$ matrix $A$, its {\em normalized numerical range} $F_N(A)$ is defined as the range of the function $f_{N,A}\colon x\mapsto (x^*Ax)/(\norm{Ax}\cdot\norm{x})$
on the complement of $\ker A$. We provide an explicit description of this set for the case when $A$ is normal or $n=2$. This extension of earlier results for particular cases of $2$-by-$2$ matrices (by Gevorgyan)
and essentially Hermitian matrices of arbitrary size (by A.~Stoica and one of the authors) was achieved due to the fresh point of view at $F_N(A)$  as the image of the Davis-Wielandt shell $\JNR(A)$ under a certain non-linear mapping $h\colon\R^3\mapsto\C$.}
\end{abstract}

\maketitle

\section{Introduction}
Throughout the paper, we denote by $\C^n$ the standard $n$-dimensional inner product space over the complex field $\C$ and by $M_n(\C)$
the algebra of all $n$-by-$n$ matrices with entries in $\C$.

The classical {\em numerical range} $F(A)$ (a.k.a. the {\em field of values}, or the {\em Hausdorff set}) of $A\in M_n(\C)$  is by definition
the set of values of the corresponding quadratic form $x^*Ax$ on the unit sphere $S\C^n := \{x \in \C^n : \|x \|=1 \}$ of $\C^n$.
Equivalently,
\[ F(A) =\left\{{(x^*Ax)}/{\norm{x}^2}\colon x\in\C^n\setminus\{0\}\right\}. \]
There are numerous papers devoted to this notion, starting with the pioneering work by Hausdorff \cite{Hau} and Toeplitz \cite{Toe18}. The Toeplitz-Hausdorff theorem states in particular that
the set $F(A)$ is convex. In fact, it is the convex hull of a certain algebraic curve $C(A)$ associated with $A$ (see e.g. \cite{Ki} or its English translation \cite{Ki08}),
sometimes called the boundary generating curve. Moreover, the eigenvalues of $A$ are the foci of $C(A)$.
Necessary and sufficient conditions on a set in $\C$ to be the numerical range of some $n$-by-$n$ matrix are known \cite{HelSpit}, though not very
easy to verify. For  our purposes, recall two basic and well known results concerning the shape of $F(A)$: for normal matrices $C(A)$ coincides with the spectrum $\sigma(A)$ of $A$, and so
$F(A)$ is nothing but the convex hull of $\sigma(A)$, while for a non-normal $A\in M_2(\C)$ it is an ellipse, and thus $F(A)$ is an elliptical disk (the Elliptical Range theorem).

Various modifications and generalization of the numerical range have been considered in the literature. Our paper is concerned with the so called {\em normalized numerical range}. Defined as
\[ F_N(A) := \left\{\frac{x^*Ax}{\|x\|\|Ax\|} \colon x \in \C^n, Ax \ne 0 \right\},\]
it was introduced in \cite{Auz}, and then further investigated in \cite{Gev04}--\cite{Gev11} and \cite{SpiSto}. Some of the elementary properties of $F_N(A)$ are similar to those of $F(A)$, and can be proved along the same lines. For convenience of reference, we collect those of them which we need in Proposition~\ref{prop:basics} below, along with brief explanations and references. Here we only note that
there is no useful analogue of the shifting property $F(A+zI)=F(A)+z$ for $F_N(A)$, which is one of the reasons why the theory of the latter is much less developed.

In particular, $F_N(A)$ was described in \cite{Gev091} for 2-by-2 normal matrices, but neither the case of normal $n$-by-$n$ matrices with $n>2$ nor the case of arbitrary $A\in M_2(\C)$
has yet been settled. More specifically, the case of $A\in M_2(\C)$ with coinciding eigenvalues, zero trace, or (at least) one eigenvalue equal to zero was tackled in \cite{Gev091}--\cite{Gev11},
but the case of a non-normal $A$ with the non-zero eigenvalues $\lambda_1\neq\pm\lambda_2$ remained open. We will deal with it in Section~\ref{s:2b2}.

On the other hand, normal matrices of arbitrary size were considered in \cite[Theorem 6.2]{SpiSto} but only when they were essentially Hermitian, i.e., in addition to $A$ being normal, the set $\sigma(A)$ was collinear
(the latter restriction of course was inconsequential for $n=2$). We will have this restriction lifted in Section~\ref{s:nor}.

A crucial ingredient used for Sections~\ref{s:2b2},~ \ref{s:nor} is the connection between $F_N(A)$ and the {\em Davis-Wieland shell} $DW(A)$ of $A$. This connection, along with
the definition of $DW(A)$ and its pertinent properties, are considered in Section~\ref{s:pre}.

\section{Preliminaries} \label{s:pre}

We begin with a proposition that collects some of the known properties of the normalized numerical range.
\begin{proposition} \label{prop:basics}
Suppose that $A \in M_n(\C)$. Then:
\begin{enumerate}[label={\alph*}.]
\item For all $z \in \NNR(A)$, $|z| \le 1$.
\item If $z \in \NNR(A)$, then $|z| = 1$ if and only if $z = \lambda/|\lambda|$ for some $\lambda \in \sigma(A)$.
\item $F_N(A)$ is unitarily invariant: $F_N(U^*AU)=F_N(A)$ for any unitary $U\in M_n(\C)$.
\item $\NNR(e^{i\theta}A) = e^{i\theta} \NNR(A)$ for all $\theta \in [0,2\pi)$.
\item $\NNR(cA) = \NNR(A)$ for all $c > 0$.
\item If $A$ is invertible, then $\NNR(A)$ is closed.
\item $\NNR(A)$ is simply connected.
\end{enumerate}
\end{proposition}

Statements (a) and (b) are simply the Cauchy-Schwarz inequality in disguise, also mentioned explicitly in \cite{Auz,Gev04}.

Statements (c)--(d) and their proofs are literally the same as those of $F(A)$. Statement (e) is different from the respective property $F(cA)=cF(A)$ but the modification is obvious.

To explain (f), as well as for some future considerations, let us introduce the function
\[ f_{N,A}\colon x\mapsto (x^*Ax)/\norm{Ax}, \quad x\in S\C^n\setminus\ker A, \]
where $\ker A$ stands, as usual, for the kernel of $A$. With this notation at hand, $F_N(A)$ is nothing but the range of
$f_{N,A}$, so we will call it the {\em normalized numerical range map}. When $A$ is invertible, the domain of $f_{N,A}$ is the whole $S\C^n$ and $F_N(A)$ is thus closed, being
the image of a compact set under a continuous mapping. This reasoning is exactly the same as for $F(A)$ (in which case it works for any $A$, invertible or not). It is worth mentioning, however, that
for non-invertible $A$ the set $F_N(A)$ may not be closed. The respective examples exist even with $A\in M_2(\C)$ and can be found in \cite{Gev09}; the closedness criterion is given by \cite[Theorem 6.4]{SpiSto}.

Property (g) was proved in \cite[Section 3]{SpiSto}, while the path-connectedness of $F_N(A)$  was established earlier in \cite[Proposition 7]{Gev04}. Note that, as opposed to $F(A)$, the set $F_N(A)$ is not necessarily convex: in particular, for normal $A\in M_2(\C)$ it was shown in \cite{Gev091} that $F_N(A)$ is a hyperbolic arc. So, path- and simple connectedness of $F_N(A)$ are by no means trivial.

In what follows, a crucial role is played by expressing the map $f_{N,A}$ as a composition of two maps.  Before describing the decomposition, let us recall some additional definitions.

The \emph{joint numerical range} (JNR for short) of a collection of $n$-by-$n$ matrices $A_1, \ldots, A_m$ is the set of $m$-tuples $W(A_1, \ldots, A_m) := \{(x^*A_1 x, \ldots, x^*A_n x) \colon x \in S\C^n \}$. As long as the matrices $A_1, \ldots, A_m$ are all Hermitian, $W(A_1,\ldots,A_m) \subset \R^m$.

Identifying $\R^2$ with $\C$ we immediately observe that $W(A_1,A_2)=F(A_1+iA_2)$ when $A_1, A_2$ are Hermitian. So, the joint numerical range is a natural generalization of the regular one. A well known result is that the joint numerical range of a family of commuting Hermitian matrices is a convex polytope. This result is analogous to the fact that the classical numerical range of a normal matrix is the convex hull of its eigenvalues. We include the statement and proof here for ease of reference.
\begin{lemma}\label{l:comher}Let $A_1,\ldots,A_m$ be an $m$-tuple of pairwise commuting Hermitian $n$-by-$n$ matrices. Then their
joint numerical range is a convex polytope.
\end{lemma}
\begin{proof}Under the conditions of the Lemma, the matrices $A_j$ can be diagonalized by a
simultaneous unitary similarity, apparently not changing their JNR. So, without loss of generality we may suppose that
$A_j$ are already diagonal: $A_j=\diag[\lambda_{j1},\ldots, \lambda_{jn}]$, $j=1,\ldots,m$. A direct computation shows then that
$W(A_1, \ldots, A_m)$ is the convex hull of the points $(\lambda_{1k},\ldots,\lambda_{mk})\in\R^m$, $k=1,\ldots,n$.
\end{proof}

The JNR of a family of $2$-by-$2$ Hermitian matrices is completely understood for any $m$, see e.g. \cite[Example 2]{GJK04} and references therein. Namely, with the exception of the situation already covered by Lemma \ref{l:comher}, $W(A_1, \ldots, A_m)$ is either a (hollow) ellipsoid,
which happens generically for $m>2$, or a (solid) ellipse (as is the case for $m=2$), depending on the rank of a certain $m$-by-$3$
matrix.

So, the joint numerical range is convex in the setting of Lemma~\ref{l:comher} but not in general. Moreover, for $m\geq 4$ and any $n$ there exist $m$-tuples of matrices in $M_n(\C)$ with non-convex JNR \cite[Proposition 2.10]{GJK04}. For our purposes, however, the case $m=3$ is important and there the JNR is convex whenever $n\geq 3$ \cite{AuT}, see also \cite[Theorem 5.4]{GJK04}.

For any $A \in M_n(\C)$, recall that $\re A := (A+A^*)/2$, and $\im A := (A-A^*)/2i$ and consider the joint numerical range $W(\re A,\im A,A^*A)$. First used in \cite{Da1,Wiel}, it is now called  the \emph{Davis-Wielandt shell} of $A$  and usually denoted $\DW(A)$. The following lemma specializes the general properties of the JNR to the case of DW.

\begin{lemma}Let $A\in M_n(\C)$.
\begin{enumerate}[label={\alph*}.]
\item If $A$ is normal with spectrum $\sigma(A)=\{\lambda_1,\ldots,\lambda_n\}$, then $DW(A)$ is the convex hull of the points $(\re\lambda_j,\im\lambda_j,\abs{\lambda_j}^2)$,
$j=1,\ldots,n$. Furthermore, each point $(\re \lambda_j, \im \lambda_j, |\lambda_j|^2)$ is an extreme point of $\DW(A)$.
\item If $n=2$ and $A$ is not normal, then $DW(A)$ is an ellipsoid.
\item If $n\geq 3$, then $DW(A)$  is convex.
\end{enumerate}
\label{l:DW}
\end{lemma}

\begin{proof}
Note that the first part of (a) follows from Lemma~\ref{l:comher}, while (b) and (c) are also stated in \cite{LiPoonSze08}, Theorems~2.2 and 2.3 respectively. It remains to prove that when $A$ is normal, $(\re\lambda_j,\im\lambda_j,\abs{\lambda_j}^2)$ is an extreme point of $\DW(A)$ for each $1\le j \le n$.  To see this, note that $\DW(A)$ is contained in the convex paraboloid $P := \{v \in \R^3 : v_1^2 + v_2^2 \le v_3 \}$. The set $P$ is strictly convex, that is, there are no non-trivial line segments in the boundary of $P$.  Since $(\re\lambda_j,\im\lambda_j,\abs{\lambda_j}^2) \in \partial P$ for each $1 \le j \le n$, it follows that no $(\re\lambda_j,\im\lambda_j,\abs{\lambda_j}^2)$ can be a non-trivial convex combination of the other $(\re\lambda_i,\im\lambda_i,\abs{\lambda_i}^2)$, $1 \le i \le n$.
\end{proof}

Observe that the normalized numerical range map $f_{N,A}$ is the composition of maps $g$ and $h$
where
\begin{equation} \label{eq:g}
g(x) := (x^*(\re A)x,x^*(\im A)x,x^*A^*Ax)
\end{equation}
and
\begin{equation} \label{eq:h}
h(v) := v_3^{-1/2}(v_1+iv_2).
\end{equation}
With this perspective, it is immediate that the normalized numerical range $F_N(A)$ is the image of the Davis-Wielandt shell $\DW(A) \subset \R^3$ under the map $h$. In fact, a more precise statement holds.

\begin{proposition} \label{prop:boundary}
Suppose $A \in M_n(\C)$. Then $\NNR(A)$ is the image of the boundary of $\JNR(A)$ under the map $h$ in \eqref{eq:h}.
\end{proposition}

\begin{proof}
We have already observed that $\NNR(A) = h(\JNR(A))$. Therefore $x+iy \in \NNR(A)$ if and only if $(xt,yt,t^2) \in \JNR(A)$ for some $t > 0$.  Consider the set $\{t > 0 : (xt,yt,t^2) \in \JNR(A) \}$. If this set is nonempty, then it has a least upper bound $t_0$ because $\JNR(A)$ is bounded.  The corresponding point $(xt_0,yt_0,t_0^2)$ will be in the boundary of $\JNR(A)$, proving the statement.
\end{proof}

It was already mentioned earlier that normalized numerical ranges are not always convex.  They do have the following property, however.

\begin{lemma} \label{lem:hyperconvex}
Let $A \in M_n(\C)$ and suppose that $p, q \in \NNR(A)$.  Then there is a hyperbola  centered at the origin such that an arc of the hyperbola connects $p$ to $q$ and is contained in $\NNR(A)$.
\end{lemma}

\begin{proof}
Since $p, q \in \NNR(A)$, there must exist $v, w \in \JNR(A)$ such that $p = h(v)$ and $q = h(w)$, where $h$ is given by \eqref{eq:h}.  If $n \ge 3$, then $\JNR(A)$ is convex by Lemma~\ref{l:DW}(c).  In that case, the line segment connecting $v$ to $w$ is contained in $\JNR(A)$.  When $n = 2$, $\JNR(A)$ is an ellipsoid, although it might not be convex.  Consider any $u \in \conv \JNR(A)$.  Let $V = \{v \in \R^3 : v_3 = u_3 \}$.  Then the image of $V \cap \JNR(A)$ under $h$ is an ellipse, and $h(u)$ is enclosed by this ellipse.  Furthermore, since this ellipse is contained in $\NNR(A)$ which is simply connected by Proposition \ref{prop:basics}, we must have $h(u) \in \NNR(A)$.

No matter what $n$ is, we conclude that the image of the line segment from $v$ to $w$ under $h$ is contained in $\NNR(A)$.  If $v_3 = w_3$, then the image of this line segment under $h$ is a line segment. If $v_3 \neq w_3$, then we may parametrize the line passing through $v$ and $w$ as
$u(t) := (a+ct^2,b+dt^2,t^2)$ for some real constants $a,b,c,d$ and a parameter $t > 0$.  Then $h(u(t)) = (a+bi)t^{-1} + (c+di)t$.  This is a real linear transformation of the hyperbolic arc $t+i/t$, and therefore $\{h(u(t)) : t > 0\}$ is a hyperbolic arc (possibly degenerate to a line or ray) and the center of the hyperbola is the origin.  The image of the line segment from $v$ to $w$ under $h$ is the portion of this hyperbolic arc that connects $p$ to $q$.
\end{proof}

\section{2-by-2 Case} \label{s:2b2}

We begin with a statement that gives an explicit equation for the Davis-Wielandt shell $\JNR(A)$ for any 2-by-2 matrix. See also \cite[Theorem 2.2]{LiPoonSze08} for an alternative description.
\begin{lemma} \label{lem:ellipsoid}
If $A \in M_2(\C)$, then $\JNR(A)$ is an ellipsoid in $\R^3$ that satisfies the equation
\begin{equation} \label{eq:ellipsoid}
a_1v_1^2 + a_2 v_2^2 + a_3 v_3^2 + a_4 v_1 v_2 + a_5 v_1 v_3 + a_6 v_2 v_3 + a_7 v_1 + a_8 v_2 + a_9 v_3 + a_{10} = 0
\end{equation}
where the coefficients are:
$$
\begin{array}{l}
a_1 = \tr(A^*A)+2\re(\det A), \\
a_2 = \tr(A^*A)-2\re(\det A),\\
a_3 = 1, \\
a_4 = 4 \im(\det A), \\
a_5 = -2 \re(\tr A), \\
a_6 = -2 \im (\tr A), \\
a_7 = -2 \re({\det A} \tr A^*), \\
a_8 = -2 \im ({\det A} \tr A^*), \\
a_9 = |\tr A|^2 - \tr(A^*A), \\
a_{10} = |\det A|^2.
\end{array}
$$
\end{lemma}
\begin{proof}
As noted in Lemma \ref{l:DW}(b), it is well known that the Davis-Weilandt shell of a 2-by-2 matrix is an ellipsoid. Therefore, $\DW(A)$ must satisfy a quadratic equation of the form \eqref{eq:ellipsoid}. Verifying the coefficients above is tedious by hand, but easy with a computer algebra system. We therefore leave it to the interested reader. It helps to apply a unitary similarity to $A$ so that it has the form
$$A = \begin{bmatrix}
\lambda_1 & c \\ 0 & \lambda_2
\end{bmatrix}$$
where $c \ge 0$.  This transformation does not change the Davis-Wielandt shell $\JNR(A)$, nor does it change the coefficients of \eqref{eq:ellipsoid} above. Note that the coordinates of a point $v \in \JNR(A)$ satisfy $v_1 = \re(x^*Ax)$, $v_2 = \im(x^*Ax)$, $v_3 = x^*A^*Ax$ for some $x \in \C^2$ with $x^*x=1$.  Therefore it suffices to verify that \eqref{eq:ellipsoid} holds for all such points, regardless of the particular unit vector $x$.
\end{proof}

With Lemma \ref{lem:ellipsoid}, we can now give a description of the normalized numerical range of a 2-by-2 matrix.

\begin{proposition} \label{prop:key}
Let $A \in M_2(\C) \backslash \{0\}$ and let
\begin{equation} \label{eq:P}
P(x,y,t) = c_0(x,y) + c_1(x,y) t + c_2(x,y) t^2 + c_3(x,y) t^3 + c_4(x,y) t^4
\end{equation}
where the coefficients $c_j(x,y)$ are the functions of $x$ and $y$ given below in terms of the coefficients $a_i$ from Lemma \ref{lem:ellipsoid}:
$$\begin{array}{l}
c_0 = c_0(x,y) = a_{10}, \\
c_1 = c_1(x,y) = a_7 x + a_8 y,\\
c_2 = c_2(x,y) = a_1 x^2 + a_2 y^2 + a_4 xy + a_9, \\
c_3 = c_3(x,y) = a_5 x + a_6 y, \\
c_4 = c_4(x,y) = a_3. \\
\end{array}$$
Then $\NNR(A)$ is the union of the family of ellipses
$$E(t) = \{x+iy : (x,y) \in \R^2, P(x,y,t) = 0\}$$
indexed by $t$ with $\sigma_1 \le t \le \sigma_2$ and $t \ne 0$ where $\sigma_1 \le \sigma_2$ are the singular values of $A$.
\end{proposition}

\begin{proof}
Let us make the following substitutions into \eqref{eq:ellipsoid}. Let $v_1 = xt$, $v_2 = yt$ and $v_3 = t^2$ where $x,y \in \R$ and $t \ge 0$. Then \eqref{eq:ellipsoid} becomes \eqref{eq:P}. For $h$ defined as in \eqref{eq:h}, we have $h(v_1,v_2,v_3) = v_3^{-1/2}(v_1 + i v_2) = x+iy$. Therefore a pair $(x,y)$ solves \eqref{eq:P} for some $t$ if and only if that pair corresponds to the image of some $(v_1,v_2,v_3) \in \JNR(A)$ under the map $h$. Note that the values of $v_3 = t^2$ in $\JNR(A)$ must fall between the eigenvalues of $A^*A$ which are the singular values of $A$, squared.  Therefore $\sigma_1 \le t \le \sigma_2$. The map $h$ is undefined when $v_3 = 0$, and so solutions to \eqref{eq:P} corresponding to $t=0$ are not part of the normalized numerical range.
\end{proof}

Normalized numerical ranges of 2-by-2 matrices always have the following symmetry property.

\begin{theorem} \label{thm:symmetry}
Let $A \in M_2(\C)$. If $A$ is invertible, then $\NNR(A)$ is symmetric across the line containing $\pm \sqrt{\det A}$.  If $A$ is rank one, then $\NNR(A)$ is symmetric across any line containing $\pm \tr A$.
\end{theorem}
\begin{proof}
Let us start with the invertible case. Note that $\NNR(A) = e^{i \theta} \NNR(e^{-i\theta}A)$ for any $\theta$ by Proposition \ref{prop:basics}(d).
Setting $\theta = \frac{1}{2}\arg(\det A)$ we may thus assume that $\det A > 0$. We can also scale $A$ by any positive constant without changing the normalized numerical range so we will assume without loss of generality that $\det A = 1$.  When $\det A = 1$, the coefficients $a_i$ from Lemma \ref{lem:ellipsoid} satisfy $a_4 = 0$, $a_7 = a_5$, $a_8 = -a_6$, and $a_{10} = a_3$.  Therefore \eqref{eq:P} becomes:
$$P(x,y,t) = (a_1 x^2 + a_2 y^2 + a_9)t^2 + a_5x (t^3+t) + a_6 y (t^3 -t)+ a_3(t^4+1).$$
Observe that $P(x,y,t) = t^4 P(x,-y,1/t)$ for all $(x,y,t) \in \R^3$ with $t > 0$. Note also that the singular values of $A$ are $\sigma_1 = \|A\|^{-1}$ and $\sigma_2 = \|A\|$. It follows that if $x+iy \in E(t)$ for some $\sigma_1 \le t \le \sigma_2$, then $x-iy \in E(t^{-1})$ and $\sigma_1 \le t^{-1} \le \sigma_2$.  Therefore $\NNR(A)$ is symmetric across the real axis in $\C$. By rotating back, the conclusion of this theorem holds for all invertible $A \in M_2(\C)$.

If $A$ is rank one, then $\det A = 0$ and the coefficients $a_i$ from Lemma \ref{lem:ellipsoid} satisfy $a_2 = a_1$, $a_4 = a_7 = a_8 = a_{10} = 0$. So \eqref{eq:P} becomes
\begin{multline*}
P(x,y,t) = (a_1 x^2 + a_1 y^2 + a_9)t^2 + (a_5 x + a_6 y) t^3 + a_3 t^4 \\
= t^2(\tr(A^*A)(x^2 + y^2) + |\tr A|^2 - \tr(A^*A) - 2(\re(\tr A) x + \im(\tr A) y) t + t^2).
\end{multline*}
In particular the ellipses $E(t)$ are all circles with centers along the line in $\C$ from the origin through $\tr A$. If $\tr A \neq 0$, then $\NNR(A)$ is symmetric across the line through $\pm \tr A$. If $\tr A = 0$, then $\NNR(A)$ will be a circle centered at the origin, and therefore will be symmetric across all lines through the origin.
\end{proof}
Note that the case when $A \in M_2(\C)$ is rank one and $\tr A = 0$ was covered in \cite[Proposition 4.1]{Gev11}.

For all rank one 2-by-2 matrices, it was shown in \cite[Proposition 3.1]{Gev11} that the boundary of the normalized numerical range is the union of two elliptical arcs. The next theorem provides the description of a larger class of $2$-by-$2$ matrices $A$ for which $F_N(A)$ has the same property.

\begin{theorem} \label{thm:twoEllipses}
Suppose $A \in M_2(\C) \backslash \{ 0 \}$. If $tr A$ and $\pm \sqrt{\det A}$ are collinear in $\C$, then the boundary of $\NNR(A)$ is the union of at most two elliptical arcs.
\end{theorem}
\begin{proof}
By rotation we may assume without loss of generality that \eq{noneg} \det A,\tr A\geq 0.\en  Indeed, for invertible $A$ let us rotate $A$ in such a way that $\det A$ becomes positive. Then the line passing through
$\pm\det A$ is simply $\R$, and $\tr A\in \R$ due to the collinearity condition. Passing from $A$ to $-A$ if needed, we can change the sign of $\tr A$ without changing $\det A$. In its turn, if $A$ is singular, then the equality $\det A=0$ persists under rotations, while $\tr A$ can be made non-negative.

We will prove now that, under conditions \eqref{noneg}, the boundary of $F_N(A)$ is given by the equations
\eq{eq:E1}
x^2 + \left( \frac{\tr(A^*A) - 2 \det A}{\tr (A^*A) - (\tr A)^2 + 2 \det A} \right)y^2 = 1
\en
for $x$ satisfying $x\tr A \ge 2 \sqrt{\det A}$, and by
\begin{multline} \label{eq:E2}
\left(\frac{\tr(A^*A)+2 \det A}{\tr(A^*A)-2\det A}\right) \left( x - \frac{2\tr A \sqrt{\det A}}{\tr(A^*A)+2 \det A} \right)^2 + y^2 \\ = \frac{\tr(A^*A) - (\tr A)^2 + 2 \det A}{\tr(A^*A) +2 \det A}
\end{multline}
for $x$ such that $x\tr A \le 2 \sqrt{\det A}$.

We will separate the proof into two cases.

{\sl Case 1.} Suppose that $A$ is invertible. By scaling, in addition to \eqref{noneg} we may assume without loss of generality that $\det A = 1$.  By Proposition \ref{prop:key}, $\NNR(A)$ is the union of the ellipses given by $P$ in \eqref{eq:P}. It will be convenient to let $Q = t^{-2}P$. Since $\tr A \in \R$ and $\det A = 1$, the coefficients $a_4, a_6, a_8 = 0$ in Lemma \ref{lem:ellipsoid}, while $a_5 = a_7 = -2 \tr A$, and $a_{10} = a_3 = 1$. Therefore
$$Q(x,y,t) = a_1 x^2 + a_2 y^2 + a_5 x (t+t^{-1}) + (t^2 + t^{-2}) + a_9 = 0.$$
Let $T := (t+t^{-1})$. Then $(t^2 + t^{-2}) = T^2 - 2$, and the equation above can be expressed as
$$Q(x,y,T) = a_1 x^2 + a_2 y^2 + a_5 x T + (T^2-2) + a_9 = 0.$$
Points on the boundary of $\NNR(A)$ are contained in the envelope of the family of ellipses $\{x+iy : Q(x,y,T) = 0\}$ indexed by $T$. This envelope consists of the points where
$$Q = \frac{\partial}{\partial T} Q = 0.$$
We compute
$$\frac{\partial}{\partial T} Q = -2 (\tr A) x  + 2T = 0,$$
which has solution $T = x\tr A$. This solution only applies if $x\tr A \ge 2$, as $T = t+t^{-1} \ge 2$ for all $\sigma_ 1 = \|A\|^{-1} \le t \le \sigma_2 = \|A \|$. Therefore, when $x \tr A \le 2$, the corresponding points on the boundary must be solutions of $P(x,y,T) = 0$ with $T =2$, or equivalently $t = 1$.  Substituting $t = 1$ into \eqref{eq:P} gives the equation
$$ (\tr(A^*A)+2) x^2 + (\tr(A^*A)-2) y^2 -4 \tr A x  + 2  + (\tr A)^2 - \tr(A^*A) = 0.$$
If we collect $x$ terms and complete the square, we get
\begin{multline*} (\tr(A^*A)+2) \left( x - \frac{2\tr A}{\tr(A^*A)+2} \right)^2 + (\tr(A^*A)-2) y^2 \\ = \frac{4 (\tr A)^2}{\tr(A^*A)+2}  - (\tr A)^2 + \tr(A^*A) - 2.\end{multline*}
Dividing through by $\tr(A^*A)-2$, we get
\begin{align*}
\left(\frac{\tr(A^*A)+2}{\tr(A^*A)-2}\right)\!\!\!\left( x - \frac{2\tr A}{\tr(A^*A)+2} \right)^2 + y^2 &= \frac{4 (\tr A)^2}{(\tr(A^*A))^2-4}  - \frac{(\tr A)^2}{\tr(A^*A)-2} + 1 \\
~ &= \frac{(\tr A)^2 (2 - \tr(A^*A))}{(\tr(A^*A))^2 - 4} + 1 \\
~ &= 1 - \frac{(\tr A)^2}{\tr(A^*A) +2}.
\end{align*}
If we replace $A$ by $A/\sqrt{\det A}$ when $\det A > 0$ in the equation above, we obtain \eqref{eq:E2}.

If, on the other hand, $x \tr A \ge 2$, then $T = x\tr A$. We can substitute $x \tr A$ for $T$ in $Q(x,y,T)$, and we obtain the following equation:
\[ (\tr(A^*A) + 2)x^2+(\tr(A^*A)-2)y^2-2(\tr A)^2 x^2+x^2(\tr A)^2-2+\abs{\tr A}^2-\tr(A^*A)\! = \! 0 \]
which simplifies to
$$(\tr(A^*A) - (\tr A)^2 + 2) x^2 + (\tr(A^*A) - 2) y^2 = \tr(A^*A) - \abs{\tr A}^2 + 2.$$
Replacing $A$ by $A/\sqrt{\det A}$ we see that this equation is equivalent to \eqref{eq:E1}.

{\sl Case 2.} $A$ is singular. Let $P$ be as in \eqref{eq:P}. Since $\det A = 0$ and $\tr A \in \R$, we have $a_2 = a_1$, $a_3 = 1$, $a_4 = a_6 = a_7 = a_8 = a_{10} = 0$, and $a_5 = \tr A$.  So
$$P(x,y,t) = (a_1 x^2 + a_1 y^2 + a_9) t^2 + a_5 x t^3 + t^4.$$
Let $Q = t^{-2}P$ and note that $E(t) = \{x+iy: Q(x,y,t) = 0 \}$ for all $0 < t \le \|A\|$. Since $\NNR(A)$ is the union of the circles $E(t)$, the boundary of $\NNR(A)$ satisfies the envelope equation
$$Q(x,y,t) = \frac{\partial}{\partial t} Q(x,y,t) = 0.$$
We compute
$$\frac{\partial}{\partial t} Q(x,y,t) =  a_5 x  + 2t = 0.$$
We may therefore substitute $-\frac{1}{2}a_5 x = x\tr A$ for $t$ in $Q$, as long as $x\tr A>0$.  We get the equation
\begin{align*}
0 &= a_1 x^2 + a_1 y^2 + a_9 - \tfrac{1}{4} a_5^2 x^2 \\
&=\tr (A^*A) (x^2 + y^2) + (\tr A)^2 - \tr A^*A - (\tr A)^2 x^2= 0. \\
&=(\tr (A^*A) - (\tr A)^2) x^2 + \tr (A^*A) y^2 + (\tr A)^2 - \tr(A^*A).
\end{align*}
This is equivalent to \eqref{eq:E1}.

If $x\tr A \le 0$, then the envelope formula no longer applies, and the boundary is determined by the circle $E(0)$.  Note that this portion of the boundary is not a subset of $\NNR(A)$, and therefore $\NNR(A)$ is not closed.  The equation for this circle is obtained by substituting $t = 0$ into the equation $Q(x,y,t) = 0$, which gives \eqref{eq:E2}.
\end{proof}

The reason for ``at most'' clause in the statement of Theorem~\ref{thm:twoEllipses} is that one of the arcs \eqref{eq:E1},\eqref{eq:E2} may degenerate into a point while the other then becomes a full ellipse. Here is when and how this happens.

\begin{corollary} \label{cor:sameAbsVal}
Suppose that $A \in M_2(\C) \backslash \{0 \}$ has eigenvalues $\lambda_1$ and $\lambda_2$ such that $|\lambda_1| = |\lambda_2|$, then $\NNR(A)$ is an elliptical disk. In the case when $\det A \ge 0$, the equation for the boundary of this ellipse is \eqref{eq:E2}. The elliptical disk is closed unless $\det A =0$, in which case $\NNR(A)$ is the open unit disk.
\end{corollary}
Note that the case when the eigenvalues of $A$ (not just their absolute values) coincide, was considered in \cite[Propositions 5.1]{Gev11}.

\begin{proof}
We may assume by rotating that $\det A \ge 0$. Then $|\lambda_1| = |\lambda_2|$ and $\lambda_1 = \overline{\lambda_2}$, so $\tr A \in \R$. Therefore Theorem \ref{thm:twoEllipses} applies. We also know that the real part of any point $x+iy \in \NNR(A)$ has absolute value at most one by Proposition \ref{prop:basics}.  Therefore $x \tr A \le |\lambda_1| + |\lambda_2| = 2|\lambda_1| = 2 \sqrt{\det A}$ for all $x+iy \in \NNR(A)$. So the boundary of the $\NNR(A)$ is given by \eqref{eq:E2}.  If $\det A \ne 0$, then $\NNR(A)$ is closed by Proposition \ref{prop:basics}(c).

If $|\lambda_1| = |\lambda_2| = 0$, then $\tr A = \det A = 0$, and \eqref{eq:E2} becomes $x^2+y^2 =1$.  In this case, no point on the boundary of $\NNR(A)$ is contained in $\NNR(A)$ since the family of ellipses $E(t)$ defined by Proposition \ref{prop:key} is expanding as $t \rightarrow 0$, with only the limiting ellipse $E(0)$ containing the boundary. Since $E(0)$ is not part of $\NNR(A)$, we see that $\NNR(A)$ is the open unit disk.
\end{proof}

There is another class of 2-by-2 matrices with elliptical normalized numerical ranges. 

\begin{theorem} \label{thm:imagTr}
Suppose that $A \in M_2(\C)$ has non-zero eigenvalues $\lambda_1, \lambda_2$ such that $\lambda_1/\lambda_2 < 0$.  Then $\NNR(A)$ is a closed elliptical disk. In the case when $\det A > 0$, the ellipse is given by the equation
\begin{equation} \label{eq:imagTrE1}
\left(\frac{\tr (A^*A) +2 \det A}{\tr(A^*A)-2 \det A - |\tr A|^2}\right) x^2 + y^2 = 1.
\end{equation}
\end{theorem}
Observe that classes of matrices considered in Theorems~\ref{thm:twoEllipses} and \ref{thm:imagTr} overlap exactly at $A\in M_2(\C)$
such that $\lambda_1=-\lambda_2\neq 0$. In this case the ellipticity of $F_N(A)$ follows also from Corollary~\ref{cor:sameAbsVal}. Furthermore,
such matrices  are traceless, and thus unitarily similar to matrices with zero main diagonal --- the case treated in \cite[Proposition 4.1]{Gev11}.
\begin{proof}
By applying a suitable complex scaling we may assume that $\det A = 1$ without changing the value of $\lambda_1/\lambda_2$. Then $\lambda_1 \lambda_2 = 1$ and $\lambda_1/\lambda_2 < 0$, so both eigenvalues must be purely imaginary and therefore $\re(\tr A) =0$.  Let $P$ be as in \eqref{eq:P} and let $Q = t^{-2} P$.  Since $\det A = 1$ and $\re(\tr A) = 0$ we have the following identities in the coefficients $a_i$ defined in Lemma \ref{lem:ellipsoid}: $a_3 = a_{10} = 1$, $a_4 = a_5 = a_7 = 0$, $a_8 = -a_6$.  Then
$$Q(x,y,t) = (a_1 x^2 + a_3 y^2 + a_9) + a_6 y(t-t^{-1}) + (t^2+t^{-2}).$$
It is convenient to let $T = t-t^{-1}$, and then
$$Q(x,y,T) = (a_1 x^2 + a_3 y^2 + a_9) + a_6 y T + (T^2 + 2) = 0,$$
which is the equation of an ellipsoid in $\R^3$.  Since the $\NNR(A)$ is the set of $x+iy$ such that $(x,y)$ solves $Q(x,y,T) = 0$ for some real $T$, it follows that $\NNR(A)$ must be an ellipse.  We now use the envelope equation
$$Q = \frac{\partial}{\partial T}Q = 0.$$
to derive a formula for this ellipse. We compute
$$\frac{\partial}{\partial T} Q(x,y,T) = a_6 y + 2 T  = 0.$$
Substituting $T = -\frac{1}{2}a_6 y= \im(\tr A) y$ into $Q(x,y,T)$ gives
\begin{align*}
0 &= (a_1 x^2 + a_3 y^2 + a_9) + a_6 y T + (T^2 + 2) \\
&= a_1 x^2 + a_3 y^2 + a_9 - \tfrac{1}{4} a_6^2 y^2 + 2 \\
&= (\tr A^*A + 2) x^2 + (\tr A^*A - |\tr A|^2 - 2) y^2  + |\tr A|^2 - \tr(A^*A) + 2.
\end{align*}
If $\det A = 1$, this is equivalent to \eqref{eq:imagTrE1}. If $\det A > 0$, then we can replace $A$ by $A/\sqrt{\det A}$ in the equation above to get \eqref{eq:imagTrE1}.
\end{proof}

\begin{remark} \label{rem:oppositeEvals}
Conditions of Theorem~\ref{thm:imagTr} hold for any real matrix with a negative determinant.
\end{remark}

The cases outlined in Theorem \ref{thm:imagTr} and Corollary \ref{cor:sameAbsVal} are the only cases where the normalized numerical range of a 2-by-2 matrix is an ellipse.
\begin{theorem}
For $A \in M_2 (\C) \backslash \{0\}$ with eigenvalues $\lambda_1$ and $\lambda_2$, the boundary of $\NNR(A)$ is an ellipse if and only if $|\lambda_1| = |\lambda_2|$ or $\lambda_1/\lambda_2 < 0$.
\end{theorem}
\begin{proof}
If $A$ is singular, then Theorem \ref{thm:twoEllipses} applies. By rotating, we may assume that $\tr A \in \R$, and then there is a single ellipse that defines the boundary of $A$ if and only if $|\lambda_1| = |\lambda_2| = 0$.

If $A$ is invertible, we may assume without loss of generality that $\det A = 1$ and $\re(\tr A) \ge 0$.  If either $\lambda_1/\lambda_2 < 0$ or $|\lambda_1| = |\lambda_2|$, then Theorem \ref{thm:imagTr} and Corollary \ref{cor:sameAbsVal} imply that the boundary of $\NNR(A)$ is an ellipse. Suppose now that $\re(\tr A) > 0$ and $|\lambda_1| \ne |\lambda_2|$.  If the boundary of $\NNR(A)$ is an ellipse, then by Theorem \ref{thm:symmetry}, the major and minor axes of that ellipse must be parallel to the real and imaginary axes, although in which order is not yet clear.  Also, the center of this ellipse must lie on the real axis.

Let $z_0 = x_0+y_0 i$ denote the rightmost point in $\NNR(A)$. By the above comments, $\im(z_0) = y_0 = 0$. Since $z_0 \in \NNR(A)$, Proposition \ref{prop:key} implies that $z_0 \in E(t_0)$ for some $\sigma_1 \le t_0 \le \sigma_2$. Since $E(t_0) \subset \NNR(A)$, $z_0$ must be the rightmost point in $E(t_0)$ as well. The ellipse $E(t_0)$ is oriented with a vertical major axis and horizontal minor axis, so if $z_0$ is the rightmost point, then the center of $E(t_0)$ must lie on the real axis.

By completing the squares on both the $x$ and $y$ variables in \eqref{eq:P}, we can derive formulas for the centers and radii of the family of ellipses $E(t) := \{x+iy : P(x,y,t) = 0\}$.  They are $x(t) + i y(t)$ where
\begin{equation} \label{eq:horizCenter}
x(t) = \frac{\re(\tr A)(t+t^{-1})}{\tr(A^*A)+2},
\end{equation}
\begin{equation} \label{eq:vertCenter}
y(t) = \frac{\im(\tr A)(t-t^{-1})}{\tr(A^*A)-2}.
\end{equation}

By \eqref{eq:vertCenter}, the center of $E(t_0)$ can only lie on the real axis if either $t_0=1$, or $\im(\tr A) = 0$.  In the later case, the boundary of $\NNR(A)$ is given by two different elliptical curves by Theorem \ref{thm:twoEllipses}, and therefore cannot be a single ellipse.  We conclude that $t_0 =1$, and we have $z_0 \in E(1)$.  Similar arguments show that the leftmost, topmost and bottommost
points of $F_N(A)$ must also be contained in $E(1)$, and therefore if $\NNR(A)$ is an ellipse, then it must equal $E(1)$.  We will show now that this cannot happen.

Using the method of Lagrange multipliers, we seek to find the points on $E(1)$ with maximal absolute value.  Equivalently, we seek to maximize $x^2+y^2$ subject to the constraint $P(x,y,1) = 0$.  We must have
$$\nabla(x^2+y^2) = \lambda \nabla P(x,y,1)$$
for some $\lambda \in \R$.  This is equivalent to
\begin{align*}
0 &= (-y,x) \cdot P(x,y,1) \\
 &= (-y,x) \cdot \nabla (b_1x^2 + b_2 y^2 + 2b_3 x + 2 + b_6) \\
 &= (-y,x) \cdot (2b_1x + 2b_3, 2 b_2 y)  \\
 &= 2 y((b_2 -b_1)x - b_3) \\
 &= 2 y (-4x + 2 \re(\tr A)).
\end{align*}
By inspection, the minimum absolute values are attained when $y =0$, and the maximum absolute values are attained when $x = \frac{1}{2} \re(\tr A)$.  We know, however, that the maximum absolute value of points in $\NNR(A)$ must be one, and that maximum occurs at the points $\lambda_1/|\lambda_1|$ and $\lambda_2/|\lambda_2|$ by Proposition \ref{prop:basics}.  For convenience, assume that $\lambda_1 = R e^{i\theta}$ with $R = |\lambda_1|$ and $\theta = \arg\lambda_1$. Then $\lambda_2 = R^{-1} e^{-i \theta}$, and $\frac{1}{2} \re(\tr A) = \frac{1}{2} (R+R^{-1}) \cos \theta$, while $\re (\lambda_1/|\lambda_1|) = \re( \lambda_2/|\lambda_2|) = \cos \theta$.  So the points on $E(1)$ with maximum absolute value correspond to $\lambda_1/\abs{\lambda_1}$ and $\lambda_2/\abs{\lambda_2}$ if and only if $R = R^{-1} =1$, and we have assumed that this is not the case since $|\lambda_1| \ne |\lambda_2|$.  This completes the proof that the only cases where the boundary of $\NNR(A)$ is an ellipse are when either $\lambda_1/\lambda_2 < 0$ or $|\lambda_1| = |\lambda_2|$.
\end{proof}

The special cases above suggest that many 2-by-2 matrices have fairly simple normalized numerical ranges.  For 2-by-2 matrices that do not fit any of these cases, it is still possible to find a polynomial equation for the boundary of the normalized numerical range.

\begin{theorem} \label{thm:resultant}
For any $A \in M_2(\C) \backslash \{0\}$, the boundary of $\NNR(A)$ satisfies a polynomial equation of degree at most 8.
\end{theorem}
\begin{proof}
Let $P(x,y,t)$ be as in \eqref{eq:P}. Any pair $(x,y)$ corresponding to $x+iy$ on the boundary of $\NNR(A)$ is contained in the envelope of $P$, therefore it satisfies
$$P(x,y,t) = \frac{\partial}{\partial t} P(x,y,t) = 0$$
for some $t$.  We can remove the variable $t$ from these two equations using the resultant of both polynomials (see \cite[Appendix 1]{Fischer} for details).  Then any pair $(x,y)$ corresponding to a point on the boundary of $\NNR(A)$ satisfies the polynomial equation
$$R(x,y):=\operatorname{resultant}\left(P(x,y,t),\frac{\partial}{\partial t} P(x,y,t), t\right) = 0.$$
This resultant is given by the determinant of the Sylvester matrix:
$$\begin{vmatrix}
c_0 & 0   & 0   & c_1   & 0     & 0     & 0 \\
c_1 & c_0 & 0   & 2 c_2 & c_1   & 0     & 0 \\
c_2 & c_1 & c_0 & 3 c_3 & 2 c_2 & c_1   & 0 \\
c_3 & c_2 & c_1 & 4 c_4 & 3 c_3 & 2 c_2 & c_1 \\
c_4 & c_3 & c_2 & 0     & 4 c_4 & 3 c_3 & 2 c_2 \\
0   & c_4 & c_3 & 0     & 0     & 4 c_4 & 3 c_3 \\
0   & 0   & c_4 & 0     & 0     & 0     & 4 c_4 \\
\end{vmatrix}$$
where $c_j = c_j(x,y)$ come from \eqref{eq:P}. Note that $c_1$ and $c_3$ are first degree polynomials in $x$ and $y$, $c_2$ is a second degree polynomial, and $c_0$ and $c_4$ are constants. By direct computation, the resultant above is at most an eighth degree polynomial in $x$ and $y$, so the theorem follows.
\end{proof}

\begin{remark}
The resultant equation in the proof of Theorem \ref{thm:resultant} for the boundary of the normalized numerical range of a matrix $A \in M_2(\C)$ gives the same equation for both $A$ and $-A$.  Therefore the solution set of the resultant equation contains the boundaries of both $\NNR(A)$ and $\NNR(-A)$.  One might wonder if the equation can be factored into two separate polynomial equations of degree at most 4 that represent the two different boundaries.  This is not possible in general, however. For example, consider the matrix
$$A = \begin{bmatrix}
2+2i & 1 \\ 0 & 1/(2+2i)
\end{bmatrix}.$$
The normalized numerical range of this matrix is shown in Figure \ref{fig:limabean}. The real roots of the resolvent equation $R(x,y) = 0$ correspond to the boundary of $\NNR(A)$ as well as the boundary of $\NNR(-A)$.  Let us suppose that the boundary of $\NNR(A)$ corresponds to a single fourth degree polynomial $Q(x,y) \in \C[x,y]$. Then the boundary of $\NNR(-A)$ would be given by the equation $Q(-x,-y)$ so $R(x,y) = Q(x,y)Q(-x,-y)$ (up to a possible scalar factor).

Since the normalized numerical range of a 2-by-2 matrix with positive determinant is symmetric across the real axis by Theorem \ref{thm:symmetry}, it follows that the real roots of $Q(c,y)$ must have this symmetry, for all constants $c \in \R$. In particular, there is an interval of values of $c$ such that $Q(c,y)$ has four real roots, as a polynomial in $y$.  Therefore, $Q(c,y)$ is a polynomial in $y$ with only even powered terms for all $c \in \R$.  In other words,
$$Q(x,y) = a_0(x) + a_2(x) y^2 + a_4(x) y^4$$
where $a_0, a_2$, and $a_4$ are polynomials in $x$ of degree at most 4.  In particular, $Q(x,-y) = Q(x,y)$ for all $x,y \in \C$.  It follows that $R(0,y) = Q(0,y) Q(0,-y) = (Q(0,y))^2$.  However, computing the value of $R(0,y)$ for this particular matrix $A$ gives
$$R(0,y) = \left(65 y^{2} + 16\right)^{2} \left(3249 y^{4} + 400 y^{2} + 576\right) /1024,$$
which is not a perfect square in the polynomial ring $\C[y]$.  Therefore, we cannot hope to express the boundary of $\NNR(A)$ using a polynomial of degree lower than 8.
\end{remark}

\begin{figure}[ht]
\begin{center}
\includegraphics[scale=0.4]{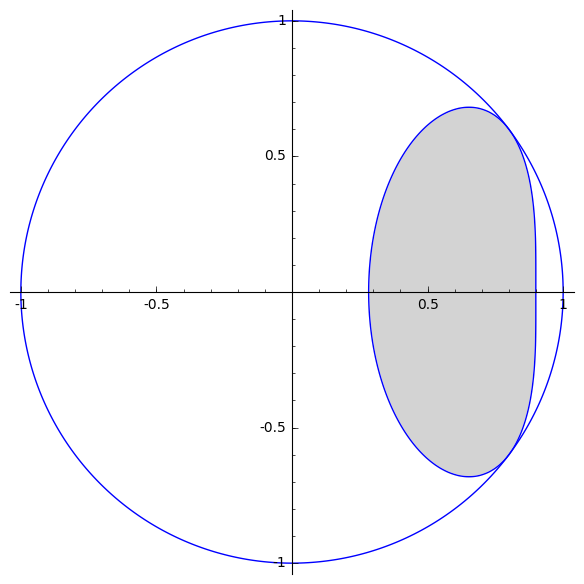}
\end{center}
\caption{The normalized numerical range of $A = [\protect\begin{smallmatrix} 4+3i & 3 \\ 0 & 9/(4+3i) \protect\end{smallmatrix}]$.}
\label{fig:convex}
\end{figure}

\begin{figure}[ht]
\begin{center}
\includegraphics[scale=0.4]{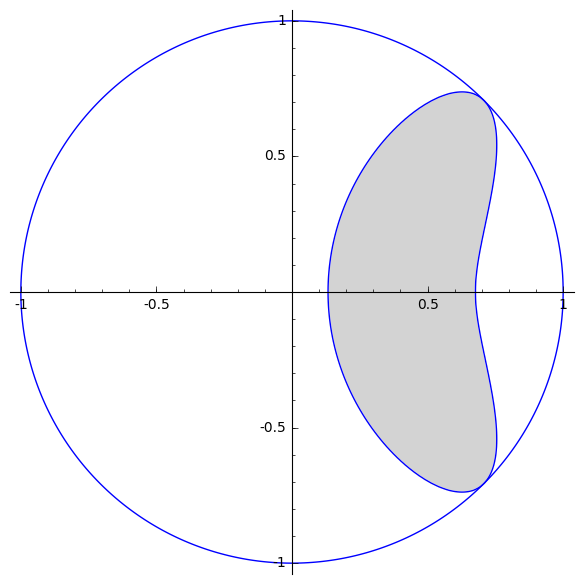}
\end{center}
\caption{For the matrix $A = [\protect\begin{smallmatrix} 2+2i & 1 \\ 0 & 1/(2+2i) \protect\end{smallmatrix}],$ $\NNR(A)$ is not convex, but the boundary is differentiable.} \label{fig:limabean}
\end{figure}

\begin{figure}[ht]
\begin{center}
\includegraphics[scale=0.4]{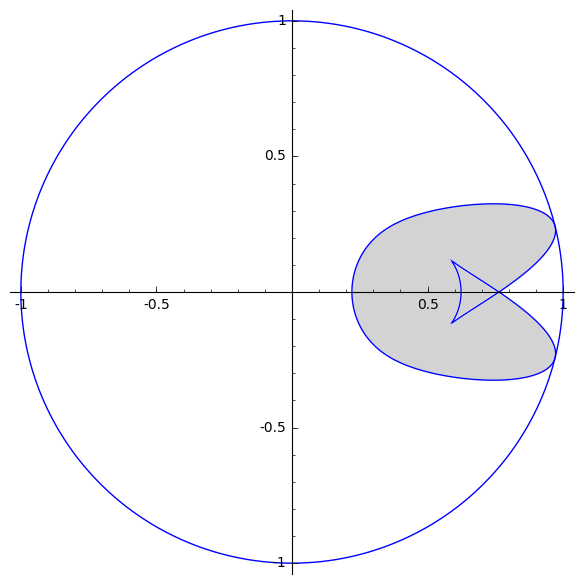}
\end{center}
\caption{For $A = [\protect\begin{smallmatrix} 4+i & 1\\ 0 & 1/(4+i) \protect\end{smallmatrix}]$, the boundary of $\NNR(A)$ is not differentiable at one point. The curves in the interior correspond to extraneous solutions of the polynomial equation for the boundary.}
\label{fig:sharp}
\end{figure}

We end this section with some examples of possible shapes of normalized numerical ranges of 2-by-2 matrices.  Figure \ref{fig:convex} is a typical convex example that is not an ellipse. Figure \ref{fig:limabean} is not convex, but has a smooth boundary.  Finally, the example in Figure \ref{fig:sharp} has a boundary that is not smooth at one point. All three of these examples have boundaries that satisfy irreducible 8th degree polynomial equations.

\section{Normal Case}\label{s:nor}

If $A$ is normal, then $\JNR(A)$ is the convex hull of the set $\{(\re \lambda_i, \im \lambda_i, |\lambda_i|^2) : \lambda_i \in \sigma(A)\}$ according to Lemma~\ref{l:DW}(a). In this section we proceed by classifying the normalized numerical ranges of normal matrices according to the dimension of the Davis-Wielandt shell. Recall \cite{Rock97} that the dimension of a convex set in $\R^n$ is the dimension of the smallest affine space containing the set.  For any $2$-by-$2$ normal matrix $A$, $\DW(A)$ is a line segment, so $\dim \DW(A) \le 1$. In fact, the condition $\dim \DW(A) \le 1$ holds if and only if $A$ is normal, with at most two distinct eigenvalues. In that case $\NNR(A)$ is a hyperbolic arc, as for $n=2$ was established in \cite[Proposition 2.1]{Gev11}, and observed to be valid for arbitrary $n$ in \cite[Theorem 5.6]{SpiSto}.

\begin{lemma} \label{lem:arcs}
Let $A \in M_n(\C) \backslash \{0\}$ be normal with at most two distinct eigenvalues $\lambda_1$ and $\lambda_2$.  If $A$ is invertible, then $\NNR(A)$ is the arc of a hyperbola centered at the origin, having endpoints $\lambda_1/|\lambda_1|$ and $\lambda_2/|\lambda_2|$, and with vertex $\frac{2 e^{i \theta} \cos \phi \sqrt{|\lambda_1 \lambda_2|}}{|\lambda_1|+|\lambda_2|}$ where $\theta = \frac{1}{2}\arg (\lambda_1 \lambda_2)$ and $\phi = \frac{1}{2} \arg(\lambda_2/\lambda_1)$. If $A$ is singular, $\NNR(A)$ is the line segment $(0,\lambda_1/|\lambda_1|]$ where $\lambda_1$ is the single non-zero eigenvalue of $A$.
\end{lemma}
Note that the case $\lambda_1=\lambda_2=0$ is trivial since then $A$, being normal, is the zero matrix, and thus $F_N(A)=\varnothing$.

Lemma \ref{lem:arcs} can also be derived from Proposition \ref{prop:key} by first scaling $A$ so that $\det A = 1$ (if $A$ is invertible), and noting that the ellipses $E(t)$ all degenerate to single points.  The centers of these points are given by \eqref{eq:horizCenter} and \eqref{eq:vertCenter}. The equation for the vertex can be derived from \eqref{eq:horizCenter} by letting $t = 1$ and rotating.

The hyperbolic arcs described in Lemma \ref{lem:arcs} only depend on the two eigenvalues of the normal matrix $A$. For convenience, we let $H(\lambda, \mu)$ denote the hyperbolic arc corresponding to a normal matrix with two distinct eigenvalues $\lambda$ and $\mu$.

In the following theorem, we completely describe the normalized numerical range of any $n$-by-$n$ normal matrix $A$ with the property that $\dim \DW(A) = 2$.      This class includes all normal matrices with three distinct eigenvalues, thus covering the case $n=3$. The latter happens to be a {\em leading special case} in Polya's terminology, to which the case of arbitrary $n$ will be reduced in the final result of this section.

When $\dim \DW(A) = 2$, there is a 2-dimensional affine subspace $V \subset \R^3$ such that $\DW(A) \subset V$. Note that this plane is vertical if and only if $A$ is {\em essentially Hermitian}, i.e.,
is normal with a collinear spectrum, and it is horizontal if and only if $A$ is a scalar multiple of a unitary matrix with at least three distinct eigenvalues.

Let us write the equation of $V$ in the form
\begin{equation} \label{eq:plane}
(\eta^T v =)\ \eta_1 v_1 + \eta_2 v_2 + \eta_3 v_3 = b
\end{equation}
where $b \in \R$ is a constant, and $\eta \in \R^3 \backslash \{ 0 \}$ is a normal vector to the plane. If $0 \in V$, then $b = 0$.  In that case $V$ is the linear span of the set $\{(\re \lambda_i, \im \lambda_i, |\lambda_i|^2): \lambda_i \in \sigma(A) \}$ which a subspace of $\R^3$.

Let $Jh(v)$ denote the Jacobian derivative of the map $h$ in \eqref{eq:h} at $v \in \R^3$ (with $v_3 > 0$). We compute
$$Jh(v) = \begin{bmatrix}
v_3^{-1/2} & 0 & -\frac{1}{2}v_1 v_3^{-3/2} \\
0 & v_3^{-1/2} & -\frac{1}{2}v_2 v_3^{-3/2} \\
\end{bmatrix}.
$$
Note that $\ker(Jh(v)) = \Span \{(v_1,v_2,2v_3)\}$.  For any $w \in \ker(Jh(v))$, the derivative of $h$ in the direction $w$ is zero. If $v+w \in V$, then $h|_V$ has a critical point at $v$.  Of course, $v+w \in V$ if and only if $w \perp \eta$.  So $h|_V$ has a critical point at $v$ if and only if
$$(\eta_1,\eta_2,\eta_3) \cdot (v_1,v_2,2v_3) = 0,$$
or equivalently by \eqref{eq:plane},
\eq{cv} b + v_3 \eta_3 = 0.\en
We will refer to the solution $v_3:=\alpha$ of \eqref{cv} (if it exists) as the \emph{critical level} of $h|_V$. This $\alpha$ plays a crucial role in the shape of $F_N(A)$.

To express the coefficients of \eqref{cv} in terms of the spectrum of $A$, pick any three distinct eigenvalues of $A$. Without loss of generality, let them be $\lambda_1, \lambda_2, \lambda_3$.
According to \eqref{eq:plane},
$$\begin{bmatrix}
\re \lambda_1 & \im \lambda_1 & |\lambda_1|^2 \\
\re \lambda_2 & \im \lambda_2 & |\lambda_2|^2 \\
\re \lambda_3 & \im \lambda_3 & |\lambda_3|^2 \\
\end{bmatrix} \,
\begin{bmatrix}
\eta_1 \\ \eta_2 \\ \eta_3
\end{bmatrix} =
\begin{bmatrix}
b \\ b \\ b
\end{bmatrix} ,
$$
and by Cramer's rule,
$$\eta_3 = b \begin{vmatrix}
\re \lambda_1 & \im \lambda_1 & 1 \\
\re \lambda_2 & \im \lambda_2 & 1 \\
\re \lambda_3 & \im \lambda_3 & 1 \\
\end{vmatrix} \left/ \begin{vmatrix}
\re \lambda_1 & \im \lambda_1 & |\lambda_1|^2 \\
\re \lambda_2 & \im \lambda_2 & |\lambda_2|^2 \\
\re \lambda_3 & \im \lambda_3 & |\lambda_3|^2 \\
\end{vmatrix} \right.$$

\noindent The bottom determinant is non-zero. Indeed, no three points $(\re \lambda_i,\im \lambda_i, \abs{\lambda_i}^2)$ with $\lambda_i \in \sigma(A)$ are colinear by Lemma \ref{l:DW}(a) and therefore they cannot be contained in a proper subspace of $\R^3$. From this, we derive the following formula for the critical level $\alpha$ of $h|_V$ when $b,\eta_3 \ne 0$.
\begin{equation} \label{eq:critlevel}
\alpha = -\begin{vmatrix}
\re \lambda_1 & \im \lambda_1 & |\lambda_1|^2 \\
\re \lambda_2 & \im \lambda_2 & |\lambda_2|^2 \\
\re \lambda_3 & \im \lambda_3 & |\lambda_3|^2 \\
\end{vmatrix} \left/ \begin{vmatrix}
\re \lambda_1 & \im \lambda_1 & 1 \\
\re \lambda_2 & \im \lambda_2 & 1 \\
\re \lambda_3 & \im \lambda_3 & 1 \\
\end{vmatrix} \right..
\end{equation}
If $b=0$ then $\alpha$ is undetermined, while when $b\neq 0$, $\eta_3 = 0$ it does not exist, and there is no critical level.

For all normal matrices with such flat Davis-Wielandt shells, we have the following description of the normalized numerical range.

\begin{theorem} \label{thm:flatnormal}
Suppose that $A \in M_n(\C)$ is normal with $m$ distinct eigenvalues $\lambda_1, \ldots, \lambda_m$ and $\dim \DW(A) = 2$. Order the eigenvalues in such a way that the edges of $\DW(A)$ connect $(\re\lambda_i,\im \lambda_i,\abs{\lambda_i}^2)$ to $(\re\lambda_{i+1},\im\lambda_{i+1},\abs{\lambda_{i+1}}^2)$ for $1 \le i \le m$ with the convention that $\lambda_{m+1}$ is identified with $\lambda_1$. Then $\NNR(A)$ is the set enclosed by the hyperbolic arcs $H(\lambda_i,\lambda_{i+1})$, $1 \le i \le m$, and possibly by one line segment that connects the points of tangency for a line bitangent to (at least) two of the hyperbolic arcs $H(\lambda_i,\lambda_{i+1})$. Such a flat portion of $\partial F_N(A)$ occurs if and only if the critical level $\alpha$ in \eqref{eq:critlevel} satisfies
\begin{equation} \label{eq:flatcondition}
\min_i |\lambda_i|^2 < \alpha < \max_i |\lambda_i|^2.
\end{equation}
In that case, the flat portion is the image of the set $\{ v \in \DW(A) : v_3 = \alpha \}$.
\end{theorem}
\begin{proof}
Since $\dim \DW(A) = 2$, $\DW(A)$ is a convex polygon contained in a 2-dimensional affine subspace of $\R^3$. By Lemma \ref{l:DW}(a), each $(\re \lambda_i, \im \lambda_i, |\lambda_i|^2)$ is an extreme point of $\DW(A)$. Therefore $\DW(A)$ has $m$ vertices and $m$ sides. We have assumed that the eigenvalues of $A$ are ordered so that the boundary edges of $\DW(A)$ correspond to pairs of eigenvalues $\lambda_i$ and $\lambda_{i+1}$ (with the convention that $\lambda_{m+1}$ is identified with $\lambda_1$). The image of the edges of $\DW(A)$ under $h$ will be the hyperbolic arcs $H(\lambda_i,\lambda_{i+1})$, $1\le i \le m$.

If $h$ is a bijection from $\JNR(A)$ onto $\NNR(A)$, then by a standard topology argument the boundary of $\JNR(A)$ will map onto the boundary of $\NNR(A)$.  In that case, the arcs $H(\lambda_i, \lambda_{i+1})$, $1 \le i \le m$, will form the boundary of $\NNR(A)$ (possibly missing the point $0 \in \partial \NNR(A)$ if $0 \in \sigma(A)$). It is possible, however, that $h$ is not one-to-one on $\JNR(A)$.

Let $V$ be the 2-dimensional affine subspace of $\R^3$ containing $\DW(A)$. If the Jacobian of $h|_V$ is full rank at $v$, then $h|_V$ has a differentiable inverse in a neighborhood of $v$ by the inverse function theorem.  On the other hand, $h|_V$ may not be invertible in a neighborhood of a critical point.  As observed in the remarks preceding the statement of the theorem, the critical points of $h|_V$ all lie on the line $V \cap \{v \in \R^3 : v_3 = \alpha \}$ where $\alpha$ is the critical level of $h|_V$.

Let $V_+ := \{v \in V: v_3 > 0 \}$ and note that $V_+$ is nonempty because it contains $\JNR(A)$. For any $v \in V_+$ we may substitute $v_1 = xt$, $v_2 = yt$, and $v_3 = t^2$ where $x,y \in \R$ and $t > 0$. With this substitution, $h(v) = x+iy$. Then, by \eqref{eq:plane}, the image of $V_+$ under $h$ is the set of points $x+iy$ in $\C$ such that $(x,y)$ satisfies
$$\eta_1 xt + \eta_2 yt + \eta_3 t^2 = b$$
or equivalently
\begin{equation} \label{eq:parallelLines}
\eta_1 x + \eta_2 y  = b t^{-1} - \eta_3 t
\end{equation}
for some $t > 0$.  Depending on the configuration of the plane $V$, the map $h|_{V_+}$ may be a bijection or not.  We have the following cases.
\begin{enumerate}
\item If $\eta_1 = \eta_2 = 0$, then $V_+ = V$ is an affine plane in $\R^3$ with constant $v_3$.  In that case, $h$ will be a one-to-one affine linear transformation from $V_+$ onto $\C$.
The hyperbolic arcs $H(\lambda_i,\lambda_{i+1})$ degenerate then into the line segments connecting $\lambda_i/\abs{\lambda_i}$ with $\lambda_{i+1}/\abs{\lambda_{i+1}}$.
\item If $\eta_3 = b = 0$, then the image of $V_+$ under $h$ is a line in $\C$ passing through the origin. This line contains $\NNR(A)$, and the theorem is trivially true.
\item If least one of $\eta_1$ or $\eta_2$ are non-zero and at least one of $b$ or $\eta_3$ are non-zero, then \eqref{eq:parallelLines} describes a family of parallel lines in $\C$ indexed by $t > 0$.  The map $h$ is one-to-one on $V_+$ if and only if $b t^{-1} - \eta_3 t$ is one-to-one.
\end{enumerate}

In the last case, if either $\eta_3 = 0$ or the critical level $\alpha = -b/\eta_3 \le 0$, then $h|_{V_+}$ is one-to-one on all of $V_+$. If that is the case, then $h$ is a bijection from $\DW(A)$ onto $\NNR(A)$. In particular, $h$ maps the relative boundary of $\DW(A)$ (i.e., the boundary of $\DW(A)$ in $V$) onto the boundary of $\NNR(A)$, which proves the theorem.

Let us consider what happens when the critical level is positive.  By \eqref{eq:parallelLines}, image of the set $V_\alpha = \{v \in V_+ : v_3 = \alpha\}$ under the map $h$ is a line $L$, while the image of $V_+$ under $h$ is the half-plane in $\C$ with boundary $L$ that does not contain the origin. All points on the boundary have a unique pre-image under $h$, while points in the open half-plane have two distinct pre-images in $V_+$. The boundary of $\NNR(A)$ may consist of the images of the edges of $\JNR(A)$ along with a flat portion that is the image of the set $\{v \in \DW(A) : v_3 = \alpha \}$ under the map $h$.  This flat portion will be a line segment contained in $L$ that connects two of the hyperbolic arcs $H(\lambda_i, \lambda_{i+1})$ and $H(\lambda_j,\lambda_{j+1})$.

Consider any line $\ell \subset V$ such that $\ell$ intersects $V_\alpha$.  Let $v$ denote the point of intersection.  Because the Jacobian of $h|_V$ is rank deficient at $v$, it follows that the tangent lines of the curves $h(\ell)$ and $L = h(V_\alpha)$ are parallel. Of course, $L$ is its own tangent line, so $h(\ell)$ is tangent to $L$ at the point of intersection. If $\ell$ is a line connecting two vertices of $\DW(A)$ and the intersection point $v$ is between the two vertices, then the corresponding hyperbolic arc $H(\lambda_i, \lambda_{i+1})$ will be tangent to $L$. This proves that if $\NNR(A)$ has a flat portion of the form described above, then it will be tangent to (at least) two of the hyperbolic arcs $H(\lambda_i, \lambda_{i+1})$.
\end{proof}

\begin{remark} The case 1. in the proof (characterized by $\eta$ being vertical) materializes exactly when $A$ is a non-zero scalar multiple of a unitary  matrix. It is therefore not surprising at all that in
this case $F_N(A)$ differs from $F(A)$ by a scalar multiple only: $F_N(A)=F(A)/\norm{A}$. The particular relation $F_N(U)=F(U)$ for unitary $U$ was observed already in \cite{Gev04}.

On the other hand, the subcase $\eta_3=0$ of 3. corresponds to essentially Hermitian $A$. The respective description of $F_N(A)$ was obtained in  \cite[Theorem 5.5]{SpiSto}. It is worth mentioning that the approach
of \cite{SpiSto} allowed for infinite-dimensional considerations, and Theorem~5.3 was derived there from a (rather more involved) result for essentially Hermitian operators on Hilbert spaces.
\end{remark}

The main theorem of this section applies to general $n$-by-$n$ normal matrices and is as follows.

\begin{theorem} \label{thm:normalMain}
Let $A \in M_n(\C)$ be normal with $m$ distinct eigenvalues $\lambda_1$, $\ldots$, $\lambda_m$.  Then $\NNR(A)$ is the set enclosed by the hyperbolic arcs $H(\lambda_i,\lambda_j)$, $1 \le i, j \le m$, $i \ne j$, along with any line segments that connect the points of tangency for a line bitangent to two arcs $H(\lambda_i,\lambda_j)$ and $H(\lambda_j,\lambda_k)$ that share an eigenvalue $\lambda_j$.  This set is closed, except possibly at the origin if $A$ is not invertible and $0 \notin \conv \{\lambda_i : \lambda_i \neq 0 \}$.
\end{theorem}

\begin{proof}
Since $A$ is normal, $\JNR(A)$ is the convex hull of $\{(\re(\lambda_i),\im(\lambda_i),|\lambda_i|^2):1 \le i \le m\}$ by Lemma \ref{l:DW}. We may assume without loss of generality that $A$ is diagonal. By Proposition \ref{prop:boundary}, $\NNR(A)$ is the image of the boundary of $\JNR(A)$ under $h$. Since $\DW(A)$ is a convex polytope, its boundary can be triangularized, that is, $\partial DW(A)$ is a union of convex triangles that correspond to Davis-Wielandt shells of principle $3$-by-$3$ submatrices of $A$. Therefore $\NNR(A)$ is the union of a finite collection of normalized numerical ranges of $3$-by-$3$ principle submatrices of $A$. Note also that $0 \in \NNR(A)$ if and only if there is a point $v = (0,0,v_3) \in \JNR(A)$.  That is true if and only if $0 \in \conv \{\lambda_i : \lambda_i \ne 0\}$.

We will now prove that if there is a line that is tangent to two of the hyperbolic arcs $H(\lambda_i, \lambda_j)$ and $H(\lambda_k, \lambda_\ell)$, and if the points of tangency are $z_1$ and $z_2$ respectively, then the line segment connecting $z_1$ and $z_2$ is contained in $\NNR(A)$.

Suppose $A$ is a 3-by-3 principle submatrix of $A$ with three distinct eigenvalues. We may assume without loss of generality that these are $\lambda_1, \lambda_2,$ and $\lambda_3$.  Suppose that the hyperbolic arcs $H(\lambda_1,\lambda_3)$ and $H(\lambda_2, \lambda_3)$ contain points $z_1$ and $z_2$ respectively, such that the line passing through $z_1$ and $z_2$ is tangent to both arcs.  Let us denote this line by $L$.  By Lemma \ref{lem:hyperconvex}, there is an arc of a hyperbola centered at 0 that connects $z_1$ to $z_2$ and is contained in $\NNR(A)$.  Since a line can pass through a hyperbola at most twice, this arc must be contained in the closed half plane with boundary $L$ that contains the origin.  At the same time, since $H(\lambda_1, \lambda_3)$ and $H(\lambda_2,\lambda_3)$ are both tangent to $L$, both of those arcs are contained in the closed half plane with boundary $L$ that does not contained the origin.  It follows that the line segment $[z_1,z_2]$ is completely enclosed by these hyperbolic arcs. Since $\NNR(A)$ is simply connected by Proposition \ref{prop:basics}, $[z_1,z_2]$ is contained in $\NNR(A)$.
\end{proof}
\color{black}

\begin{remark}
If $A \in M_n(\C)$ is normal and $0 \in \NNR(A)$, then the boundary of $A$ will consist solely of hyperbolic arcs $H(\lambda_i, \lambda_j)$ where $\lambda_i, \lambda_j \in \sigma(A)$.  From the proof of Theorem \ref{thm:normalMain}, if two hyperbolic arcs $H(\lambda_i, \lambda_j)$ and $H(\lambda_j, \lambda_k)$ are both tangent to a line $L$, then the line segment in $L$ that connects the two points of tangency must be contained in $\NNR(A)$.  If $0 \in \NNR(A)$, then any ray from $0$ passing through a point on that line segment will also pass through one of the hyperbolic arcs $H(\lambda_i, \lambda_j)$ or $H(\lambda_j, \lambda_k)$.  As noted in \cite[Theoerem 3.3]{SpiSto}, the intersection of $\NNR(A)$ with any ray from the origin is connected. This means that points on $L \cap \NNR(A)$ can only be on the boundary of $\NNR(A)$ if they are contained in one of the hyperbolic arcs $H(\lambda_i, \lambda_j)$. This is also essentially true for any singular normal matrix, although in that case $0$ may also be an additional boundary point.
\end{remark}

\begin{figure}[ht]
\begin{center}
\begin{tikzpicture}
\begin{scope}[scale=0.21]
\clip (-11,-11) rectangle (15,11);
\fill[fill=gray!40, even odd rule] (5.3,-1.836) circle (5.046) (10,20.6) circle (20.6);
\fill[white] (0,0) circle (2.45);
\fill[white] (0,-10) arc (-90:180:10) -- (-11,12) -- (16,12) -- (16,-11) -- cycle;
\draw[<->] (-11,0) -- (11,0);
\draw[<->] (0,-11) -- (0,11);
\draw[dashed] (0,0) circle (10);
\draw[dashed] (0,0) circle (2.45);
\filldraw (1.732,1.732) circle (4pt) node[below,scale=0.75] {$\lambda_1$};
\filldraw (10,0) circle (4pt) node[below left,scale=0.75] {$\lambda_3$};
\filldraw (3,4) circle (4pt) node[above right,scale=0.75] {$\lambda_2$};
\draw[blue] (10,20.6) circle (20.6);
\draw[blue] (5.3,-1.836) circle (5.046);
\end{scope}

\begin{scope}[xshift=5cm,yshift=-2.0cm,scale=3.6]
\clip (-0.1,1.3) rectangle (1.3,-0.1);
\draw[->] (-1,0) -- (1.2,0);
\draw[->] (0,-1) -- (0,1.2);
\draw[blue] (0,0) circle (1);

\shadedraw[draw=blue,left color=gray!70, right color=gray!40] (0.6, 0.8) -- (0.5998, 0.79882) -- (0.5996, 0.79763) -- (0.59941, 0.79645) -- (0.59922, 0.79527) -- (0.59904, 0.79408) -- (0.59886, 0.7929) -- (0.59869, 0.79172) -- (0.59852, 0.79054) -- (0.59836, 0.78936) -- (0.5982, 0.78817) -- (0.59805, 0.78699) -- (0.5979, 0.78581) -- (0.59776, 0.78464) -- (0.59762, 0.78346) -- (0.59749, 0.78228) -- (0.59736, 0.7811) -- (0.59724, 0.77993) -- (0.59712, 0.77875) -- (0.59702, 0.77758) -- (0.59691, 0.7764) -- (0.59682, 0.77523) -- (0.59673, 0.77406) -- (0.59665, 0.77289) -- (0.59657, 0.77172) -- (0.5965, 0.77055) -- (0.59644, 0.76939) -- (0.59638, 0.76822) -- (0.59633, 0.76706) -- (0.59629, 0.7659) -- (0.59626, 0.76474) -- (0.59624, 0.76358) -- (0.59622, 0.76242) -- (0.59621, 0.76127) -- (0.59621, 0.76012) -- (0.59622, 0.75897) -- (0.59623, 0.75782) -- (0.59626, 0.75667) -- (0.5963, 0.75553) -- (0.59634, 0.75439) -- (0.59639, 0.75325) -- (0.59646, 0.75212) -- (0.59653, 0.75099) -- (0.59661, 0.74986) -- (0.59671, 0.74873) -- (0.59681, 0.74761) -- (0.59693, 0.74649) -- (0.59706, 0.74537) -- (0.5972, 0.74426) -- (0.59735, 0.74316) -- (0.59751, 0.74205) -- (0.59751, 0.74205) -- (0.59868, 0.73448) -- (0.59984, 0.72692) -- (0.60101, 0.71935) -- (0.60218, 0.71178) -- (0.60334, 0.70421) -- (0.60451, 0.69664) -- (0.60567, 0.68907) -- (0.60684, 0.68151) -- (0.608, 0.67394) -- (0.60917, 0.66637) -- (0.61033, 0.6588) -- (0.6115, 0.65123) -- (0.61267, 0.64367) -- (0.61383, 0.6361) -- (0.615, 0.62853) -- (0.61616, 0.62096) -- (0.61733, 0.61339) -- (0.61849, 0.60582) -- (0.61966, 0.59826) -- (0.62083, 0.59069) -- (0.62199, 0.58312) -- (0.62316, 0.57555) -- (0.62432, 0.56798) -- (0.62549, 0.56042) -- (0.62665, 0.55285) -- (0.62782, 0.54528) -- (0.62899, 0.53771) -- (0.63015, 0.53014) -- (0.63132, 0.52257) -- (0.63248, 0.51501) -- (0.63365, 0.50744) -- (0.63481, 0.49987) -- (0.63598, 0.4923) -- (0.63715, 0.48473) -- (0.63831, 0.47716) -- (0.63948, 0.4696) -- (0.64064, 0.46203) -- (0.64181, 0.45446) -- (0.64297, 0.44689) -- (0.64414, 0.43932) -- (0.6453, 0.43176) -- (0.64647, 0.42419) -- (0.64764, 0.41662) -- (0.6488, 0.40905) -- (0.64997, 0.40148) -- (0.65113, 0.39391) -- (0.6523, 0.38635) -- (0.65346, 0.37878) -- (0.65463, 0.37121) -- (0.6558, 0.36364) -- (0.6558, 0.36364) -- (0.66006, 0.34065) -- (0.66515, 0.32018) -- (0.67087, 0.30176) -- (0.67705, 0.28504) -- (0.6836, 0.26975) -- (0.69042, 0.25569) -- (0.69745, 0.24267) -- (0.70464, 0.23056) -- (0.71195, 0.21926) -- (0.71935, 0.20865) -- (0.72681, 0.19867) -- (0.73432, 0.18924) -- (0.74185, 0.18031) -- (0.74939, 0.17184) -- (0.75694, 0.16377) -- (0.76449, 0.15607) -- (0.77202, 0.14871) -- (0.77953, 0.14166) -- (0.78701, 0.13489) -- (0.79447, 0.12839) -- (0.8019, 0.12213) -- (0.8093, 0.1161) -- (0.81666, 0.11027) -- (0.82398, 0.10464) -- (0.83126, 0.09919) -- (0.8385, 0.09391) -- (0.8457, 0.0888) -- (0.85286, 0.08383) -- (0.85998, 0.079) -- (0.86706, 0.07431) -- (0.87409, 0.06974) -- (0.88108, 0.06529) -- (0.88803, 0.06095) -- (0.89494, 0.05672) -- (0.90181, 0.05258) -- (0.90863, 0.04855) -- (0.91541, 0.04461) -- (0.92215, 0.04075) -- (0.92885, 0.03697) -- (0.93552, 0.03328) -- (0.94214, 0.02966) -- (0.94872, 0.02611) -- (0.95526, 0.02263) -- (0.96176, 0.01922) -- (0.96823, 0.01587) -- (0.97466, 0.01259) -- (0.98105, 0.00936) -- (0.9874, 0.00618) -- (0.99372, 0.00307) -- (1.0, 0.0) -- (1.0, 0.0) -- (0.99348, 0.00806) -- (0.98692, 0.01625) -- (0.98031, 0.02456) -- (0.97366, 0.03301) -- (0.96697, 0.04159) -- (0.96023, 0.05032) -- (0.95344, 0.05919) -- (0.94661, 0.06822) -- (0.93973, 0.07741) -- (0.9328, 0.08677) -- (0.92582, 0.0963) -- (0.91879, 0.10601) -- (0.91171, 0.11591) -- (0.90457, 0.12601) -- (0.89738, 0.13631) -- (0.89013, 0.14683) -- (0.88283, 0.15757) -- (0.87547, 0.16854) -- (0.86805, 0.17976) -- (0.86056, 0.19124) -- (0.85302, 0.20299) -- (0.84541, 0.21502) -- (0.83774, 0.22735) -- (0.83, 0.24) -- (0.82219, 0.25298) -- (0.81431, 0.26632) -- (0.80637, 0.28002) -- (0.79834, 0.29413) -- (0.79025, 0.30865) -- (0.78207, 0.32362) -- (0.77382, 0.33906) -- (0.76549, 0.35501) -- (0.75707, 0.3715) -- (0.74857, 0.38857) -- (0.73999, 0.40627) -- (0.73131, 0.42463) -- (0.72255, 0.44372) -- (0.71369, 0.4636) -- (0.70474, 0.48432) -- (0.6957, 0.50596) -- (0.68656, 0.52862) -- (0.67732, 0.55238) -- (0.66799, 0.57736) -- (0.65855, 0.60368) -- (0.64902, 0.63148) -- (0.63939, 0.66095) -- (0.62967, 0.69227) -- (0.61986, 0.72569) -- (0.60997, 0.76149) -- (0.6, 0.8);

\shadedraw[draw=blue,left color=gray!30, right color=gray!60] (0.70711, 0.70711) -- (0.70009, 0.70493) -- (0.69356, 0.70307) -- (0.68748, 0.7015) -- (0.6818, 0.70018) -- (0.67649, 0.69909) -- (0.67152, 0.69822) -- (0.66686, 0.69754) -- (0.66249, 0.69703) -- (0.65839, 0.69669) -- (0.65454, 0.6965) -- (0.65091, 0.69644) -- (0.6475, 0.69651) -- (0.64428, 0.6967) -- (0.64125, 0.69699) -- (0.63839, 0.69739) -- (0.6357, 0.69787) -- (0.63315, 0.69845) -- (0.63075, 0.6991) -- (0.62849, 0.69982) -- (0.62635, 0.70062) -- (0.62432, 0.70148) -- (0.62241, 0.7024) -- (0.62061, 0.70337) -- (0.6189, 0.7044) -- (0.61729, 0.70548) -- (0.61577, 0.7066) -- (0.61434, 0.70777) -- (0.61298, 0.70898) -- (0.6117, 0.71022) -- (0.6105, 0.7115) -- (0.60936, 0.71281) -- (0.60829, 0.71416) -- (0.60729, 0.71553) -- (0.60634, 0.71693) -- (0.60545, 0.71836) -- (0.60462, 0.71981) -- (0.60383, 0.72129) -- (0.6031, 0.72279) -- (0.60241, 0.7243) -- (0.60177, 0.72584) -- (0.60118, 0.7274) -- (0.60062, 0.72897) -- (0.60011, 0.73056) -- (0.59964, 0.73216) -- (0.5992, 0.73378) -- (0.5988, 0.73541) -- (0.59843, 0.73705) -- (0.59809, 0.73871) -- (0.59779, 0.74037) -- (0.59751, 0.74205) -- (0.59751, 0.74205) -- (0.59868, 0.73448) -- (0.59984, 0.72692) -- (0.60101, 0.71935) -- (0.60218, 0.71178) -- (0.60334, 0.70421) -- (0.60451, 0.69664) -- (0.60567, 0.68907) -- (0.60684, 0.68151) -- (0.608, 0.67394) -- (0.60917, 0.66637) -- (0.61033, 0.6588) -- (0.6115, 0.65123) -- (0.61267, 0.64367) -- (0.61383, 0.6361) -- (0.615, 0.62853) -- (0.61616, 0.62096) -- (0.61733, 0.61339) -- (0.61849, 0.60582) -- (0.61966, 0.59826) -- (0.62083, 0.59069) -- (0.62199, 0.58312) -- (0.62316, 0.57555) -- (0.62432, 0.56798) -- (0.62549, 0.56042) -- (0.62665, 0.55285) -- (0.62782, 0.54528) -- (0.62899, 0.53771) -- (0.63015, 0.53014) -- (0.63132, 0.52257) -- (0.63248, 0.51501) -- (0.63365, 0.50744) -- (0.63481, 0.49987) -- (0.63598, 0.4923) -- (0.63715, 0.48473) -- (0.63831, 0.47716) -- (0.63948, 0.4696) -- (0.64064, 0.46203) -- (0.64181, 0.45446) -- (0.64297, 0.44689) -- (0.64414, 0.43932) -- (0.6453, 0.43176) -- (0.64647, 0.42419) -- (0.64764, 0.41662) -- (0.6488, 0.40905) -- (0.64997, 0.40148) -- (0.65113, 0.39391) -- (0.6523, 0.38635) -- (0.65346, 0.37878) -- (0.65463, 0.37121) -- (0.6558, 0.36364) -- (0.6558, 0.36364) -- (0.65528, 0.36707) -- (0.65479, 0.37057) -- (0.65433, 0.37412) -- (0.65388, 0.37775) -- (0.65346, 0.38144) -- (0.65307, 0.38521) -- (0.65271, 0.38904) -- (0.65238, 0.39296) -- (0.65207, 0.39695) -- (0.6518, 0.40103) -- (0.65156, 0.40519) -- (0.65136, 0.40944) -- (0.6512, 0.41378) -- (0.65107, 0.41822) -- (0.65098, 0.42276) -- (0.65094, 0.4274) -- (0.65094, 0.43214) -- (0.65098, 0.437) -- (0.65108, 0.44198) -- (0.65123, 0.44708) -- (0.65143, 0.4523) -- (0.65169, 0.45766) -- (0.65201, 0.46316) -- (0.6524, 0.46881) -- (0.65285, 0.4746) -- (0.65338, 0.48056) -- (0.65398, 0.48669) -- (0.65467, 0.49299) -- (0.65543, 0.49948) -- (0.65629, 0.50617) -- (0.65725, 0.51306) -- (0.65831, 0.52017) -- (0.65948, 0.52751) -- (0.66077, 0.53509) -- (0.66218, 0.54294) -- (0.66373, 0.55105) -- (0.66542, 0.55946) -- (0.66726, 0.56818) -- (0.66927, 0.57723) -- (0.67146, 0.58664) -- (0.67384, 0.59642) -- (0.67642, 0.60661) -- (0.67923, 0.61723) -- (0.68229, 0.62832) -- (0.6856, 0.63992) -- (0.68921, 0.65206) -- (0.69313, 0.66479) -- (0.6974, 0.67817) -- (0.70204, 0.69226) -- (0.70711, 0.70711);
\end{scope}
\end{tikzpicture}
\end{center}

\caption{(Left) When $\lambda_1 = \sqrt{3}+\sqrt{3}i$ and $\lambda_3 = 10$, the shaded regions show the values of $\lambda_2$ where $\NNR(\operatorname{diag}(\lambda_1,\lambda_2,\lambda_3))$ has a flat portion. In particular, $\lambda_2 = 3+4i$ is in the shaded region. (Right) $\NNR(A)$ for $A = \operatorname{diag}(\sqrt{3}+i\sqrt{3},3+4i,10)$. Note the flat portion of the boundary connecting the hyperbolic arcs $H(\sqrt{3}+i\sqrt{3},10)$ and $H(3+4i,10)$.}
\label{fig:circles}
\end{figure}
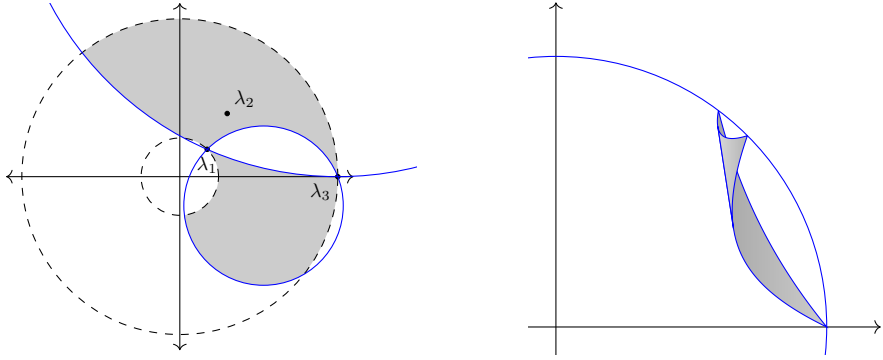

\begin{example}
Consider the matrix
$$A = \begin{bmatrix}
\sqrt{3}+i \sqrt{3}  & 0 & 0 \\ 0 & 3+4i & 0 \\ 0 & 0 & 10
\end{bmatrix}.$$
The normalized numerical range of this matrix appeared in \cite[Figure 2]{Gev04}.  It is not obvious from the figure there, but can be verified using \eqref{eq:flatcondition} that $\NNR(A)$ has a flat portion on its boundary (see Figure \ref{fig:circles}).

The condition in \eqref{eq:flatcondition} can be interpreted geometrically.  Suppose that $\lambda_1$ and $\lambda_3$ are fixed, while $\lambda_2$ is allowed to vary, as long as $|\lambda_1| \le |\lambda_2| \le |\lambda_3|$.  When the right and left-hand inequalities in \eqref{eq:flatcondition} are replaced with equalities, we get equations for two different circles in the complex plane, both passing through $\lambda_1$ and $\lambda_3$.  The set of values of $\lambda_2$ where the normalized numerical range of the diagonal matrix $\diag(\lambda_1,\lambda_2,\lambda_3)$ has a flat portion (not corresponding to a hyperbolic arc $H(\lambda_i,\lambda_j)$) is given by those $\lambda_2$ that are contained in one, but not both, of these circles. For example, in the matrix $A$ above, $\lambda_1 = \sqrt{3}+i\sqrt{3}$, $\lambda_2 = 3+4i$, and $\lambda_3 = 10$.  The eigenvalue $\lambda_2$ falls inside one of the two circles but not the other, as shown in Figure \ref{fig:circles}.
\end{example}

\section*{\refname}
\nocite{Gev06}
\bibliography{../master}
\bibliographystyle{plain}

\end{document}